\documentclass[12pt,aps,tightenlines,notitlepage,superscriptaddress]{revtex4-1}
\usepackage{amsmath}
\usepackage{amssymb}
\usepackage{bm}
\usepackage[mathscr]{eucal}
\usepackage{theorem}
\usepackage{graphicx}
\usepackage{psfrag}
\usepackage{subfigure}
\usepackage{color}
\usepackage{wasysym}
\include{graphix}
\usepackage{framed}
\usepackage{enumitem}%
\usepackage{lipsum}
\usepackage{float}
\usepackage{MnSymbol}
\usepackage{url}

\usepackage{geometry}
\graphicspath{{figs/}}

\newcommand{\imag}{\mathrm{i}}
\newcommand{\mathe}{\mathrm{e}}
\newcommand{\mathd}{\mathrm{d}}

\newcommand{\mydim}{n}
\newcommand{\mysmall}{\epsilon}
\newcommand{\myreg}{\mathcal{R}}
\newcommand{\myflow}{\varphi_{\delta}}

\newcommand{\vecx}{\bm{x}}

\newcommand{\myre}{\mathrm{Re}}

\newcommand{\myspin}{\mathcal{S}}

\newcommand{\myconst}{\varphi_*}

\newcommand{\mypsivec}{\pmb{\psi}}
\newcommand{\myonevec}{\mathbf{1}}
\newcommand{\myfvec}{\pmb{f}}
\newcommand{\mydiff}{\pmb{\mathcal{D}}}
\newcommand{\myjac}{\pmb{\mathcal{J}}}
\newcommand{\myS}{\pmb{\mathcal{S}}}
\newcommand{\mypsivecn}{\pmb{\psi}^n}
\newcommand{\mypsivecnpone}{\pmb{\psi}^{n+1}}
\newcommand{\myFvec}{\pmb{F}}

\newcommand{\myk}{j}
\newcommand{\mychat}{a}

\newtheorem{remark}{Remark}
\newtheorem{theorem}{Theorem}
\newenvironment{proof}[1][Proof]{\begin{trivlist}
\item[\hskip \labelsep {\bfseries #1}]}{\end{trivlist}}

\newcommand{\myqed}{\nobreak \ifvmode \relax \else
      \ifdim\lastskip<1.5em \hskip-\lastskip
      \hskip1.5em plus0em minus0.5em \fi \nobreak
      \vrule height0.75em width0.5em depth0.25em\fi}

\begin{document}
\title{Travelling-wave spatially periodic forcing of asymmetric binary mixtures}
\author{Lennon \'O N\'araigh\footnote{Email address: \texttt{onaraigh@maths.ucd.ie}}}
\address{School of Mathematics and Statistics, University College Dublin, Belfield, Dublin 4, Ireland}


\date{\today}

\begin{abstract}
We study travelling-wave spatially periodic solutions of a forced Cahn--Hilliard equation.  This is a model for phase separation of a binary mixture, subject to  external forcing.  We look at arbitrary values of the mean mixture concentration, corresponding to asymmetric mixtures (previous studies have only considered the symmetric case).   We characterize in depth one particular solution which consists of an oscillation around the mean concentration level, using a range of techniques, both numerical and analytical.  We determine the stability of this solution to small-amplitude perturbations.  Next, we use methods developed elsewhere in the context of shallow-water waves to uncover a (possibly infinite) family of multiple-spike solutions for the concentration profile, which linear stability analysis demonstrates to be unstable.  Throughout the work, we perform thorough parametric studies to outline for which parameter values the different solution types occur.
\end{abstract}

\maketitle


\section{Introduction}
\label{sec:intro}

When a binary fluid in which both components are initially well mixed undergoes rapid cooling below a critical temperature, both phases spontaneously separate to form domains rich in the fluid's component parts. The domains expand over time in a phenomenon known as coarsening~\cite{Bray_advphys}.  The coarsening can be modified and controlled by external influences, including stirring by an externally-imposed flow~\cite{naraigh2015flow,shear_Bray}, or external heating of the fluid.  The focus of this work is on the latter.  In particular, we look at the Ludwig--Sorret effect, whereby fluctuations in the binary-fluid concentration are produced via thermal diffusion~\cite{craig2004thermal}.  

Theoretical and experimental study of phase separation is justified on several grounds.  
Phase separation of binary liquids has many practical applications, especially in the fabrication of electro-optical devices~\cite{kukadiya2009ludwig} and thin-film coating~\cite{wilczek2015modelling}.  
The phenomenon is of interest from the scientific point of view, where it appears in surprising contexts, for instance, in the evaporation of sessile droplets of binary  mixtures, where the anisotropic droplet curvature produces a correspondingly asymmetric pattern of phase separation~\cite{saenz2017dynamics}.
Finally,  simplified theoretical models of phase separation (such as the Cahn--Hilliard equation~\cite{CH_orig}) have interesting mathematical properties which have faciliated a complete classification of the model solutions in certain circumstances~\cite{Elliott_Zheng}.

The basic mathematical model of phase separation used in this work is the Cahn--Hilliard equation, wherein a single scalar concentration field $C(\vecx,t)$ can be used to fully characterize the binary mixture~\cite{CH_orig}.  As such, a concentration level $C=\pm 1$ indicates phase separation of the mixture into one or other of its component parts, while $C=0$ denotes a perfectly mixed state.  It is further assumed that the system is in  the spinodal region of the thermodynamic phase space, where the well-mixed state is energetically unfavorable. Consequently, the free energy for the mixture can be modeled as $F[C]=\int_\Omega \left[(1/4)(C^2 -1)^2 +(1/2)\gamma|\nabla C|^2\right] \mathd^\mydim x$, where the first term promotes demixing and the second term smooths out sharp gradients in transition zones between demixed regions; also, $\gamma$ is a positive constant, $\Omega$ is the container where the binary fluid resides, and $\mydim$ is the dimension of the space. The twin constraints of mass conservation and energy minimization suggest a gradient-flow dynamics for the evolution of the concentration: $\partial_t C = \nabla\cdot\left[D(C)\nabla(\delta F/\delta C)\right]$, where $\delta F/\delta C$ denotes the functional derivative of the free energy and $D(C)\geq 0$ is the mobility function, assumed for simplicity in this work to be a positive constant.  As such, the basic model equation reads
\begin{equation}
\frac{\partial C}{\partial t}=D\nabla^2\left(C^3-C-\gamma\nabla^2 C\right).
\label{eq:ch_basic}
\end{equation}

The basic mathematical model~\eqref{eq:ch_basic} can be modified in numerous ways to take account of the various external influences that can be imposed on the physical system so as to control the phase separation.  In this work, we focus on the Ludwig--Sorret effect, whereby concentration fluctuations are induced by an externally-imposed temperature gradient.  Mathematically, this amounts to adding a source term to the right-hand side of Equation~\eqref{eq:ch_basic}.
To develop a concise mathematical description of such controlled phase separation, we focus for simplicity on a one-dimensional version of Equation~\eqref{eq:ch_basic}, with a source term that takes the form of a travelling wave:
\begin{equation}
\frac{\partial C}{\partial t}=D\partial_{xx}\left(C^3-C-\gamma\partial_{xx} C\right)+f_0k\cos[k(x-vt)],
\label{eq:ch_tw}
\end{equation}
where $k$ is the forcing wave number, $f_0k$ is the forcing amplitude,  and $v$ is the velocity of the travelling wave.  Also, the fluid container is taken as $\Omega=\mathbb{R}$ in an abstract setting, although this will be restricted in what follows.  Such travelling-wave forcing has been studied before for symmetric binary mixtures wherein the spatial average $\langle C\rangle$ of the concentration is zero~\cite{weith2009traveling}.  Therefore, the main contribution of the present work is to extend this prior work by looking at $\langle C\rangle\neq 0$.  Indeed, we demonstrate that $\langle C\rangle$ is a  crucial parameter which can be used to control the phase separation, along with  $(f_0,v,D,k,L)$.

\section{Problem Statement and Methodology}
\label{sec:problem}

We seek solutions of Equation~\eqref{eq:ch_tw} that inherit the spatiotemporal structure of the forcing term.  As such, we seek spatially-periodic travelling wave solutions
\begin{equation}
C(x,t)=\psi(\eta),\qquad \eta=x-vt,\qquad \psi(\eta+L)=\psi(\eta),
\label{eq:ch_twx}
\end{equation}
where $L=2\pi/k$ is the periodicity of the forcing.  The trial solution~\eqref{eq:ch_twx} is substituted into Equation~\eqref{eq:ch_tw} to produce
\begin{equation}
-v \frac{\mathd\psi}{\mathd\eta}=D\frac{\mathd^2}{\mathd\eta^2}\left(\psi^3-\psi-\gamma\frac{\mathd^2\psi}{\mathd\eta^2}\right)+f_0k\cos(k\eta).
\label{eq:ch_tw1}
\end{equation}
Equation~\eqref{eq:ch_tw1} is integrated once and the periodic boundary conditions are used to determine the resulting constant of integration.  This yields
\begin{equation}
\gamma D \frac{\mathd^3\psi}{\mathd\eta^3}=
D\frac{\mathd}{\mathd\eta}\left(\psi^3-\psi\right)+v\left(\psi-\langle \psi\rangle\right)+f_0\sin(k\eta),
\label{eq:ch_tw2}
\end{equation}
where
\[
\langle \psi\rangle=\frac{1}{L}\int_0^L \psi(\eta)\mathd\eta
\]
is the mean value of the concentration.  Therefore, the problem statement and the main aim of this paper is to characterize the solutions of Equation~\eqref{eq:ch_tw2}.

A key special-case solution of Equation~\eqref{eq:ch_tw2} occurs when $f_0=v=\langle\psi\rangle=0$, whereupon Equation~\eqref{eq:ch_tw2} can be integrated to give the solution
\begin{equation}
\psi(\eta)=\tanh\left(\frac{\eta-\eta_0}{\sqrt{2\gamma}}\right),
\label{eq:mytanh}
\end{equation}
where $\eta_0$ is an arbitrary constant.  This is a known equilibrium solution of the full temporally-evolving equation~\eqref{eq:ch_tw} with $f_0=0$.  Indeed, the dynamics of Equation~\eqref{eq:ch_tw} (with $f_0=0$) is such that  
 an arbitrary mean-zero initial condition will rapidly evolve into a concentration profile comprising extended regions where $C\approx \pm 1$, separated by a $\tanh$-like transition zone such as~\eqref{eq:mytanh}.  The extended regions $C\approx \pm 1$ subsequently interact and merge (subject to the constraint that the mean concentration is conserved).  In a nutshell, this is the coarsening dynamics of the (unforced) Cahn--Hilliard equation.  Although this brief study of the classic $\tanh$-solution~\eqref{eq:mytanh} would appear incidental to the present study on the corresponding forced Cahn--Hilliard equation, the $\tanh$ profile is of key importance to constructing multiple-spike solutions of the forced Cahn--Hilliard equation, which we develop in this work.

We notice further that Equation~\eqref{eq:ch_tw2} has a large set of different parameters.  Throughout this work, we will employ various techniques to reduce the number of parameters down to a minimum of independent parameters.   This will enable us to carry out a comprehensive parameter study outlining the different possible solution behaviours as the independent parameters are varied.  As a starting-point of this reduction, we make the following remark:
\begin{remark}
\label{thm:symmetry}
If $\psi(\eta)$ is a smooth $L$-periodic solution of Equation~\eqref{eq:ch_tw1} with mean $c_0$, then $\widehat{\psi}(\eta)=-\psi(\eta+L/2)$ is a smooth $L$-periodic solution with mean $-c_0$.
\end{remark}
This can be shown by direct computation.  As a result, it suffices in any parameter study to focus on the case with $\langle \psi\rangle\geq 0$, since cases with $\langle \psi\rangle<0$ can be obtained by symmetry.  
We furthermore focus on a parameter regime where $\mysmall=\gamma/L^2\rightarrow 0$, which is physically representative of binary-fluid systems~\cite{LowenTrus}.  Two limiting cases of Equation~\eqref{eq:ch_tw2} then occur:
\paragraph*{1.  The Regular Limit}  In this limit, a regular perturbation theory $\psi=\psi_0(\eta)+\epsilon\psi_1(\eta)+\cdots$ is admissible, in which case Equation~\eqref{eq:ch_tw2} reduces to a first-order ODE, in the lowest order in the perturbation theory:
\begin{equation}
0=
D\frac{\mathd}{\mathd\eta}\left(\psi_0^3-\psi_0\right)+v\left(\psi_0-\langle \psi\rangle\right)+f_0\sin(k\eta),
\label{eq:ch_reg}
\end{equation}
Notice that Remark~\ref{thm:symmetry} carries over to this case.  We refer to Equation~\eqref{eq:ch_reg} as the `reduced-order model'.
\paragraph*{2.  The Singular Limit} In this limit, the regular perturbation theory~\eqref{eq:ch_reg} breaks down, higher-order derivatives become important, and the `full model' (i.e. Equation~\eqref{eq:ch_tw2}) is required.

Examples (based on numerical solutions) of both limits are shown in Figure~\ref{fig:examples}.
\begin{figure}
  \subfigure[$\,\, f_0=0.12$]{\includegraphics[width=0.48\textwidth]{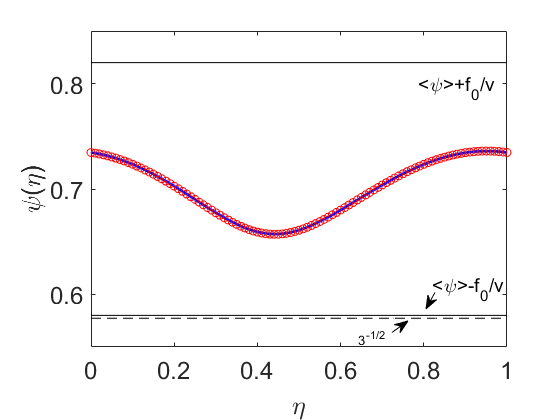}}
  \subfigure[$\,\, f_0=0.24$]{\includegraphics[width=0.48\textwidth]{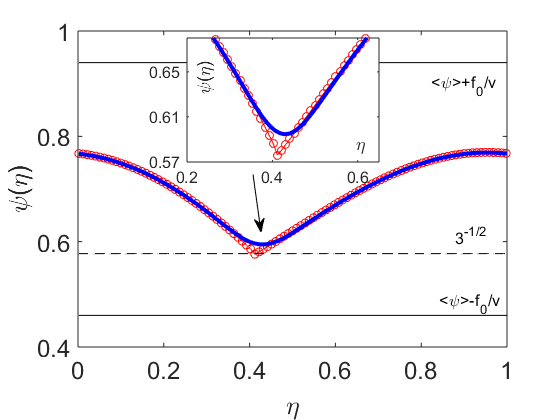}}
  \caption{Sample $L$-periodic numerical solutions of the full model  (Equation~\eqref{eq:ch_tw2}, solid line, with $\epsilon=10^{-4}$) and the reduced-order model (Equation~\eqref{eq:ch_reg}, circles). 
  The following parameters are the same in both panels: $\langle \psi\rangle=0.7$, $v=D=L=1$, $k=2\pi$.  The inset in panel (b) is an enlargement of the main figure which shows the formation of the cusp in more detail.
	Details of the numerical method are provided below at the foot of this section (Section~\ref{sec:problem}) and also in~\ref{app:methodology}.
  }
  \label{fig:examples}
\end{figure}
Panel (a) corresponds to a case where the regular perturbation theory holds and the resulting reduced first-order model~\eqref{eq:ch_reg}
are valid.  Hence, for an appropriate (small) value of $\epsilon$, there is very close agreement between the full third-order problem and the first-order model solution (the first-order model corresponds to $\epsilon\rightarrow 0$).
On the other hand, panel (b)  corresponds to a case where the regular perturbation theory is on the verge of breaking down, such that a cusp in the concentration profile forms at $\psi=1/\sqrt{3}$.  
In the region close to the cusp, there is significant disagreement between the third-order and first-order models.  Outside of this region, the agreement
between the two models remains close.

Summarizing, the plan of the paper is as follows.
We first of all obtain \textit{a priori} necessary conditions on the parameter values $(f_0,\langle \psi\rangle,v,D,k,L)$ for which the reduced-order model is valid.  This is accomplished below in Section~\ref{sec:reduced} using methods from Functional Analysis (specifically, Fixed Point Theorems).
We explore further parameter regions (i.e. beyond where the above \textit{a priori} theory is valid) where the reduced order model is still valid.  This is done using numerical solutions of Equation~\eqref{eq:ch_tw2} and~\eqref{eq:ch_reg}.
We characterize the linear stability of the reduced-order model.  This is accomplished below in Section~\ref{sec:stability} using linear stability analysis.  As such, a solution of Equation~\eqref{eq:ch_tw} consisting of a reduced-order travelling wave and a time-dependent perturbation is studied.  The perturbation is computed using Bloch's Theorem / Floquet Analysis.  A further  \textit{a priori} condition for the growth rate of the perturbation to be negative is computed from upper bounds of various integrals. 
Finally, we explore the last remaining regions of the parameter space wherein the reduced-order model is no longer valid.  As such, we consider the full model equation~\eqref{eq:ch_tw2}. We construct various travelling-wave solutions numerically, and determine their stability.  Several solution types emerge in this manner, only one of which resembles the solution type (and corresponding concentration profile) found in the reduced-order model.  
The numerically-constructed travelling-wave solutions of the full model are carefully checked against temporally-evolving numerical simulations (TENS) of equation~\eqref{eq:ch_tw}.  For stable parameter cases, the travelling waves emerge from temporal simulations with random initial conditions.
For the full model, we construct a `flow-pattern map' outlining the parameter regimes where the various travelling-wave solutions are found.  This approach is inspired by the literature on Multiphase Flow for Engineering applications~\cite{Hewitt1982}, where distinct flow regimes are mapped out as a function of the flow parameters.

The computational methodology is therefore severalfold.  We use a Newton iterative solver with linesearch to compute numerical travelling-wave solutions of Equation~\eqref{eq:ch_tw2} and~\eqref{eq:ch_reg}.  Physical intuition gleaned in Sections~\ref{sec:reduced} and~\ref{sec:full} is used to construct initial conditions for the solver.
A complementary approach to finding the travelling-wave solutions is also used for the reduced-order model.  In this complementary approach, we solve Equation~\eqref{eq:ch_reg} numerically using an eighth-order accurate Runge-Kutta scheme~\cite{ode87}. The periodic boundary conditions are imposed numerically using a `shooting' method.
Finally,  we use temporally-evolving numerical simulations (TENS) of equation~\eqref{eq:ch_tw}.  These simulations are accomplished using a pseudospectral numerical method based on Reference~\cite{naraigh2007bubbles}.
These methods are documented and validated extensively in~\ref{app:methodology}.

\section{The Reduced-Order Model}
\label{sec:reduced}

We determine parameter regimes wherein the reduced-order~\eqref{eq:ch_reg} is valid.  The approach is twofold: we use numerical solutions to map out a parameter space where the reduced-order model is valid.  Then, using analytical techniques in certain limiting cases, we characterize these solutions rigorously.


\subsection{Periodic solutions -- quantitative analysis}
\label{eq:analytical}

We first of all look at the case where $f_0\rightarrow 0$, such that a second application of regular perturbation theory may again be used, with
\begin{equation}
\psi=\langle \psi \rangle + f_0\varphi_1(\eta)+O\left( f_0^2\right),
\label{eq:ch_small0}
\end{equation}
where $\varphi_1$ satisfies 
\begin{equation}
D\frac{\mathd\varphi_1}{\mathd \eta}=-\frac{v}{3\langle \psi\rangle^2-1}\varphi_1
-\frac{1}{3\langle \psi\rangle^2-1}\sin(k\eta).
\label{eq:ch_small}
\end{equation}
We further require that $3\langle \psi\rangle^2-1\neq 0$.
Equation~\eqref{eq:ch_small} is a standard first-order linear ordinary differential equation.  The solution is made up of two parts.  The homogeneous part can be written as $\myconst\mathe^{-\kappa \eta}$, where $\kappa=(v/D)(3\langle \psi\rangle^2-1)^{-1}$, and $\myconst$ is a constant of integration.  The particular integral can be written as $\alpha \sin (k\eta)+\beta\cos(k\eta)$, where $\alpha$ and $\beta$ are constants chosen such that Equation~\eqref{eq:ch_small} is satisfied.  The particular integral is intrinsically $L$-periodic, whereas the homogeneous solution is $L$-periodic only when $\myconst=0$.
%
%
Hence,
\begin{equation}
\varphi_1(\eta)=\frac{\kappa^2}{k^2+\kappa^2}\left[\frac{k}{\kappa}\cos(k\eta)-\sin(k\eta)\right].
\label{eq:explicit_small1}
\end{equation}
%
%
%

The condition $3\langle \psi\rangle^2-1\neq 0$ in the limiting case $f_0\rightarrow 0$ has wider significance when $f_0$ is finite.   As such, we assume quite generally that Equation~\eqref{eq:ch_reg} has a smooth solution, which is equivalently the solution to the equation
\begin{equation}
D\frac{\mathd\psi}{\mathd \eta}=-\frac{v\left(\psi-\langle \psi\rangle\right)}{3\psi^2-1}-\frac{f_0}{3\psi^2-1}\sin(k\eta).
\label{eq:ch_reg1}
\end{equation}
At an extreme point (maximum / minimum), we have $\mathd \psi/\mathd\eta=0$, hence
\begin{equation}
\psi_{\mathrm{max/min}}=\langle \psi\rangle-\frac{f_0}{v}\sin\left(k\eta_{\mathrm{max}/\mathrm{min}}\right),
\end{equation}
hence
\begin{equation}
\langle\psi\rangle - \frac{f_0}{v} \leq \psi_{\mathrm{min}}\leq \psi(\eta)\leq 
\psi_{\mathrm{max}}\leq \langle\psi\rangle + \frac{f_0}{v}.
\label{eq:psimaxmin}
\end{equation}
On the other hand, Equation~\eqref{eq:ch_reg1} has a singularity at $\psi=\pm 1/\sqrt{3}$.  However, by controlling the maximum and the minimum of $\psi$, 
the trajectory of the differential equation~\eqref{eq:ch_reg1} may avoid the singularity.  
This control can be achieved in any one of the following three parameter cases:
\begin{equation}
\left.\begin{aligned}
\text{Case 0: } &1/\sqrt{3} < \langle \psi\rangle-(f_0/v),\\
\text{Case 1: } &\langle \psi\rangle-(f_0/v) > -1/\sqrt{3} \text{ and }\langle \psi\rangle+(f_0/v)<1/\sqrt{3},\\
\text{Case 2: }&\langle \psi\rangle+(f_0/v)<-1/\sqrt{3}.
\end{aligned}
\,\,\right\}
\label{eq:cases}
\end{equation}

Although Equation~\eqref{eq:cases} suggests that $f_0/v$ is a pertinent parameter group of fundamental relevance to the basic equation~\eqref{eq:ch_reg}, this is not so: a quick inspection of Equations~\eqref{eq:ch_reg} reveals that the independent parameters in the problem are $v/D$, $f_0/D$, and $\langle \psi\rangle$.  This is also confirmed by the asymptotic analysis~\eqref{eq:ch_small0}, where the correction at $O(f_0^2)$ (not shown) demonstrates these dependencies clearly.  The following further remark clarifies the number of independent parameters necessary for a complete parameter study:
\begin{remark}
\label{rem:params}
The value of $k$ is fixed as $k=2\pi$ by the  choice to fix the wavelength of the travelling-wave forcing term as the problem lengthscale.  Also,  $D$ can be set to unity without loss of generality, as this amounts to rescaling time in the full spatiotemporal model~\eqref{eq:ch_tw}.   As such, we set $D=1$ and $k=2\pi$ throughout the remainder of this study.
\end{remark}

The condition~\eqref{eq:cases} is a sufficient condition whereby the reduced-order model~\eqref{eq:ch_reg} has a regular solution.  However, it is not a necessary condition.   Therefore, in order to map out comprehensively the parameter regimes wherein Equation~\eqref{eq:ch_reg} has a periodic travelling-wave solution, we solve Equation~\eqref{eq:ch_reg} numerically using a `shooting' technique (see Section~\ref{sec:problem} and~\ref{app:methodology} for details).  We do not solve Equation~\eqref{eq:ch_reg} directly; instead we solve for $X=\psi^3-\psi$:
\begin{equation}
D\frac{\mathd X}{\mathd \eta}+v\left(\psi_j(X)-\langle \psi\rangle\right)+f_0\sin(k\eta)=0,
\label{eq:dX}
\end{equation}
where
\[
\psi_j(X)=\frac{2}{\sqrt{3}}\cos\left[ \tfrac{1}{3}\cos^{-1}\left(\tfrac{3\sqrt{3}}{2}X\right)-\frac{2\pi j}{3}\right],\qquad j=0,1,2.
\]
Equation~\eqref{eq:dX} possesses complex discontinuous solutions in parts of parameter space where the basic equation~\eqref{eq:ch_reg} is singular.  Therefore, the absence of any such complex solutions indicates that the basic equation is regular.

Motivated by these considerations, we have scanned a 
$\left(\langle \psi\rangle,f_0\right)$ parameter (sub-)space for fixed $v$, mapping out regions where Equation~\eqref{eq:dX} possesses a real-valued smooth solution, for the various values of $j$. 
%
The results of the scan are shown in Figure~\ref{fig:parameter_space_cj}.  The symmetry of the figure under $\langle \psi\rangle\rightarrow -\langle \psi\rangle$ is a consequence of Theorem~\ref{thm:symmetry}.
\begin{figure}
	\centering
		\includegraphics[width=0.6\textwidth]{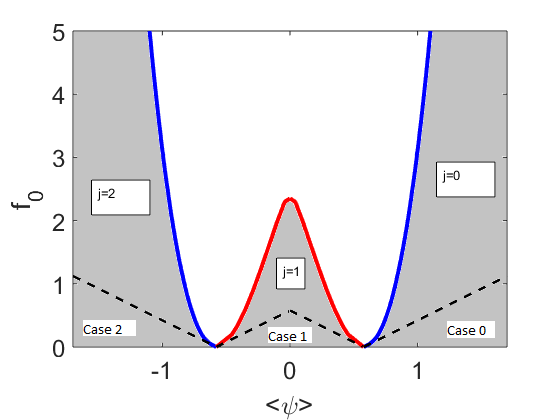}
		\caption{The parameter subspace $\left(\langle \psi\rangle,f_0\right)$ at fixed $v=1$.  Shaded regions correspond to parameter values where the basic equation~\eqref{eq:ch_reg} has precisely one smooth solution; unshaded regions correspond to parameter values where only a singular (complex-valued) solution exists.  Areas  underneath the broken lines correspond to the corresponding Cases 0--2 
mapped out by Equation~\eqref{eq:cases}.}
	\label{fig:parameter_space_cj}
\end{figure}
For any particular value of $\left(\langle \psi\rangle,f_0\right)$ in the figure, there is at most one regular periodic solution of Equation~\eqref{eq:ch_reg}, corresponding to a definite single value of $j$ in Equation~\eqref{eq:dX}.  These regular solutions correspond to the shaded regions of the parameter space in the figure; correspondingly, unshaded regions represent  parameter values where only a singular (complex-valued) solution exists.   The areas underneath the broken lines in the figure represent the special cases mapped out in Equation~\eqref{eq:cases} -- these areas are clearly a subset of the shaded regions.  Thus, the cases in Equation~\eqref{eq:cases} give a subset of all possible smooth solutions, and hence, Equation~\eqref{eq:cases} gives a sufficient but not a necessary condition for the existence of smooth solutions.  Finally, each regular solution is checked and it is confirmed that $|3\psi^2-1|>0$ in each case.  Thus, the possibility of regular solutions with a cosmetic singularity in the governing equation~\eqref{eq:ch_reg} is ruled out.  

\subsection{Periodic solutions -- qualitative analytical results}

It is of interest to look more closely at the parameter regime covered by Equation~\eqref{eq:cases}, as rigorous analysis can be used in this instance to characterise the periodic solutions.  As such, in this section
we use Brouwer's and Banach's Fixed Point Theorems to show rigorously
that a unique periodic solution exists for all values of $f_0/v$ covered by Equation~\eqref{eq:cases}.

\begin{theorem}
\label{thm:exist}
Suppose that any one of the cases in Equation~\eqref{eq:cases} holds.  Then Equation~\eqref{eq:ch_reg1} has at least one $L$-periodic solution.
\end{theorem}

\begin{proof}

The idea of the proof is to construct a scalar-valued function $f$ of a single real variable that maps an interval $I$ of allowed initial values $\psi_0$ at $\eta=0$ to corresponding final values $\psi(L)$
at $\eta=L$.  The function $f$ will be constructed and it will be shown that $f$ has at least one fixed point.  
For definiteness, consideration is given to Case 1, where the
interval of allowed initial conditions is 
\[
I=[a,b]=\left[\langle\psi\rangle-(f_0/v),\langle \psi\rangle+(f_0/v)\right],
\]
such that $\psi_0\in I$.  The other cases (Cases 0 and 2) are  very similar.   
The construction of the function $f$ is shown intuitively in Figure~\ref{fig:f_fn}.
\begin{figure}
\centering
  \includegraphics[width=0.6\textwidth]{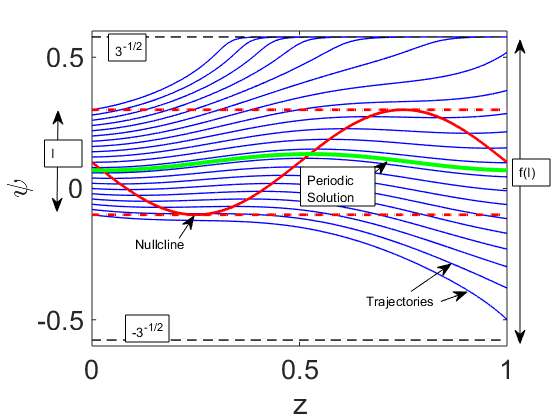}
  \caption{The construction of the $f$-function mapping the interval $I$ to $[-1/\sqrt(3),1/\sqrt{3}]$
   (Case 1).
   Shown also is the nullcline $\psi=\langle \psi\rangle-(f_0/v)\sin(kz)$ across which $\mathd\psi/\mathd\eta$
   changes sign.
   }
  \label{fig:f_fn}
\end{figure}
In what follows, the function $f$ is constructed more formally.

The starting-point of the construction of the function $f$ 
is a regularized version of Equation~\eqref{eq:ch_reg1}:
\begin{equation}
D\frac{\mathd\psi}{\mathd \eta}=-\myreg(3\psi^2-1,\delta)\left[v\left(\psi-\langle \psi\rangle\right)+f_0\sin(k\eta)\right],
\label{eq:ch_regreg}
\end{equation}
where $\myreg(s,\delta)=s/(s^2+\delta^2)$, where $\delta$ is a small but positive mollifier.  
Thus, $\myreg(s,\delta)\rightarrow 1/s$ as $\delta \rightarrow 0$ and for $s\neq 0$, 
such that Equation~\eqref{eq:ch_regreg} reduces to Equation~\eqref{eq:ch_reg1} provided one 
avoids the singular points $\psi=\pm 1/\sqrt{3}$.
Equation~\eqref{eq:ch_regreg} can further be viewed as a two-dimensional system of autonomous differential equations in an enlarged phase space:
\begin{subequations}
\begin{equation}
\frac{\mathd}{\mathd \eta}\left(\begin{array}{c}z \\ \psi\end{array}\right)=
%
\left(\begin{array}{c} 1 \\ 
-\myreg(3\psi^2-1,\delta)\left[v\left(\psi-\langle \psi\rangle\right)+f_0\sin(k\eta)\right]
\end{array}\right),
\end{equation}
with initial conditions
\begin{equation}
z=z_0,\qquad \psi=\psi_0,\qquad \text{ at }\eta=0.
\end{equation}%
\label{eq:ch_system}%
\end{subequations}%
Equations~\eqref{eq:ch_system} reduce back to the single non-autonomous differential equation~\eqref{eq:ch_regreg} when $z_0=0$.

Solutions of Equation~\eqref{eq:ch_system} are embedded in the flow $\myflow$, namely the map
\begin{subequations}
\begin{eqnarray}
\myflow:\mathbb{R}^2\times \mathbb{R} &\rightarrow & \mathbb{R}^2,\nonumber\\
        \left( (z_0,\psi_0), \eta \right)&\mapsto & \myflow(z_0,\psi_0,\eta),
\end{eqnarray}
where
\begin{equation}
\myflow(z_0,\psi_0,\eta)=\left(\begin{array}{c}z=z_0+\eta\\ \psi(\eta)\end{array}\right)
\end{equation}%
\label{eq:flowdef}%
\end{subequations}%
and where $\psi(\eta)$ satisfies Equation~\eqref{eq:ch_system} with initial condition $\psi(0)=\psi_0$ and $z_0=0$.  
Using the flow $\myflow$, the function $f$ can now be prescribed explicitly; it is
\begin{equation}
f:I\rightarrow \mathbb{R},\qquad \psi_0\mapsto f(\psi_0),
\end{equation}
where
\begin{equation}
f(\psi_0)=\left(\begin{array}{cc}0& 0\\ 0 & 1\end{array}\right)\myflow(0,\psi_0,L).
\end{equation}
Referring to Figure~\ref{fig:f_fn}, and to
 the structure of the system of differential equations~\eqref{eq:ch_system}, it is clear that
 Equation~\eqref{eq:ch_system} has no fixed points, periodic orbits, limit cycles, etc.  
 As such, any trajectory starting at $(\psi_0\in I,z_0=0)$ will pass through $z=L$.  Furthermore, 
 since the right-hand side of Equation~\eqref{eq:ch_system} is a smooth function, the flow $\myflow$ is 
 also a smooth function, and hence, $f(\psi_0)$ is a continuous function on the closed domain $I$.
 
The function $g(x)=f(x)-x$ is now introduced.  
Notice that $a$ lies to one side of the nullcline
$\psi=\langle \psi\rangle-(f_0/v)\sin(kz)$ 
across which $\mathd\psi/\mathd\eta$ changes sign.  Thus, $\psi(\eta)$ is a decreasing function along the trajectory starting at $\psi(0)=a$.  
Thus, $g(a)<0$.  Similarly, $g(b)>0$.  Hence, $g(x)$ changes sign on the interval $(a,b)$ and 
so $g(x)$ has at least one zero, $g(x_*)=0$.  Therefore, 
$f(x)$ has at least one fixed point $f(x_*)=x_*$.  Hence, the following special solution of 
Equation~\eqref{eq:ch_regreg}
\begin{equation}
D\frac{\mathd\psi}{\mathd \eta}
=-\myreg(3\psi^2-1,\delta)\left[v\left(\psi-\langle \psi\rangle\right)+f_0\sin(k\eta)\right],
\qquad \eta\in (0,L),\qquad \psi(0)=x_*,
\label{eq:ch_exist}
\end{equation}
is the required periodic solution (albeit of the regularized equation), since $\psi(L)=x_*$.

Finally, the solution of Equation~\eqref{eq:ch_exist} is bounded above and below in the range given by Equation~\eqref{eq:psimaxmin}.
In this range, the solution of Equation~\eqref{eq:ch_exist} is independent of the (small) value of $\delta$,
since $\psi$ never approaches the singularities at $\pm 1/\sqrt{3}$ -- these singular points being the only place where the 
regularization has any effect.  As such, a solution of Equation~\eqref{eq:ch_exist} with $\delta\rightarrow 0$ 
gives the required periodic solution to the unregularized problem.  \myqed

\end{proof}

\begin{remark}
The proof of the existence of a fixed point of the function $f(x)$ can be viewed as a particular application
of Brouwer's Fixed Point Theorem.
\end{remark}

It can be further shown that the fixed point of the function $f(x)$ is unique in Case 1. 
\begin{theorem}
Suppose that Case 1 of Equation~\eqref{eq:cases} holds.  Then Equation~\eqref{eq:ch_reg1} has exactly 
one $L$-periodic solution.
\end{theorem}

The idea of the proof is to look at the  magnitude of $f'(x)$ on an appropriate sub-interval 
$J\subset I$ with $x_*\in J$.  Here, $x_*$ is the fixed point already identified in Theorem~\ref{thm:exist}.
As such, the starting-point is the inequality
\[
\left|f(\psi_1)-f(\psi_0)\right|\leq \max_{x\in J}|f'(x)||\psi_1-\psi_0|,\qquad \psi_0,\psi_1\in J,
\]
If $|f'(x)|<1$ for all $x\in J$, then $f(x)$ is a contraction mapping,
such that Banach's Fixed Point Theorem can be used to demonstrate the uniqueness of the fixed point.  Otherwise, if $|f'(x)|>1$
for all $x\in I$ (or for all $x$ in a sub-interval $J\subset I$ containing a known fixed point $x_*$), then 
the inverse map can be constructed such that $f^{-1}$ is a contraction mapping, to which Banach's Fixed Point Theorem can again be applied.
More formally, the proof proceeds as follows.

\begin{proof}
Consider again the basic first-order non-autonomous differential equation~\eqref{eq:ch_regreg},
recalled here as
\begin{equation}
D\frac{\mathd\psi}{\mathd \eta}=F_\delta(\psi,z),\qquad F_\delta(\psi,z)=-\myreg(3\psi^2-1,\delta)\left[v\left(\psi-\langle \psi\rangle\right)+f_0\sin(k z)\right],
\label{eq:ch_regregF}
\end{equation}
where $z=z_0+\eta$ and $z_0=0$.  
The corresponding flow $\myflow$ is again given by Equation~\eqref{eq:flowdef}.  
We further define
\begin{equation}
F(\psi,z)=\lim_{\delta\rightarrow 0}F_\delta(\psi,z)
=-\frac{1}{3\psi^2-1}\left[v\left(\psi-\langle \psi\rangle\right)+f_0\sin(k z)\right],
\end{equation}
provided the limit exists.
Using
standard results concerning the flow $\myflow$, it can be shown that the derivative of $f(\psi_0)$ is given by
\begin{equation}
f'(\psi_0)=\exp\left(\int_0^L \frac{\partial F_\delta}{\partial \psi}\big|_{\psi(\eta)}\mathd \eta\right),
\label{eq:formal}
\end{equation}
where the trajectory $\psi(\eta)$ starts from $\eta=0$ and $\psi(0)=x_*$ and 
returns to $x_*$ at $\eta=L$; in other words, $x_*$ is a fixed point as identified previously in Theorem~\ref{thm:exist}.  Moreover, the periodic trajectory of interest 
remains far from the singular points at $\pm 1/\sqrt{3}$, such that the limit $\delta\rightarrow 0$ in Equation~\eqref{eq:formal} can be taken at $\psi_0=x_*$,
and Equation~\eqref{eq:formal} becomes
\begin{equation}
f'(\psi_0)=
\lim_{\delta\rightarrow 0}\exp\left(\int_0^L \frac{\partial F_\delta}{\partial \psi}\big|_{\psi(\eta)}\mathd \eta\right)
=\exp\left(\int_0^L \frac{\partial F}{\partial \psi}\big|_{\psi(\eta)}\mathd \eta\right),
\label{eq:formalx}
\end{equation}
where Equation~\eqref{eq:formalx} is valid in an open sub-interval $J\subset I$ containing $x_*$, and where
\begin{equation}
\frac{\partial F}{\partial \psi}=
\frac{3v\psi^2-6\psi \left[v\langle \psi\rangle-f_0\sin (k z)\right]+v}{(3\psi^2-1)^2}.
\label{eq:dFdc}
\end{equation}

In what follows, it is helpful to specify the subinterval $J$ explicitly.  Since the unregularized ODE has singularities at $\psi=\pm 1/\sqrt{3}$, the corresponding function $f:I\rightarrow \mathbb{R}$ is not continuous on the full interval $I$.
By restricting the function $f$ to the sub-interval $J=f^{-1}(I)$ for the present purposes,
continuity (and differentiability) on the corresponding open subinterval is regained; moreover,
\[
f:J\rightarrow I,\qquad \psi_0\mapsto f(\psi_0)
\]
is a monotone-increasing function on the restricted domain $J$, and is therefore invertible on the same.  This idea is shown 
schematically in Figure~\ref{fig:f_fn2}.
\begin{figure}
\centering
  \includegraphics[width=0.3\textwidth]{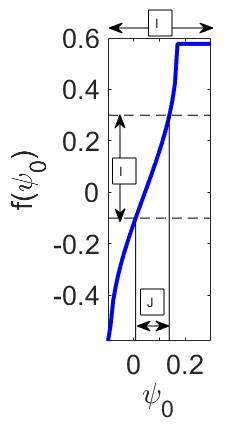}
  \caption{The restriction of the domain of the function $f(\eta)$ to $J=f^{-1}(I)$, the inverse-image of the interval $I$ (Case 1).
   }
  \label{fig:f_fn2}
\end{figure}
By invertibility, $x_*\in f^{-1}(I)$ also, as required for the proof.

The sign of $\partial F/\partial \psi$ in Equation~\eqref{eq:dFdc} now determines the behaviour of the derivative of the function $f$.  
The behaviour of the numerator is key: the numerator
is a quadratic function of $\psi$, and is positive-definite provided the sign of the appropriate discriminant is positive, i.e.
\[
1-\tfrac{1}{\sqrt{3}}\frac{1}{\left[\langle \psi\rangle-(f_0/v)\sin (k z)\right]^2}>0.
\]
The requirement is satisfied if
\[
\left| \langle \psi\rangle - (f_0/v)\sin (k z)\right|<\frac{1}{\sqrt{3}}.
\]
This is precisely the parameter range described by Case 1.
In this case, $\partial F/\partial \psi>0$, and $f'(x)>1$ by Equation~\eqref{eq:formal}, for all $x\in J$.  
Correspondingly, the inverse $f^{-1}:I\rightarrow J$ is a contraction mapping, whose derivative has the explicit form
\begin{equation}
\left(f^{-1}\right)'(\psi_0)=\exp\left(-\int_0^L \frac{\partial F}{\partial \psi}\big|_{\psi(\eta)}\mathd \eta\right),
\label{eq:formal2}
\end{equation}
where now $\psi(\eta)$ is a trajectory such that $\psi(L)=\psi_0$.
Thus, $|f^{-1}(x)|<1$ for all $x\in I$, and so $f^{-1}:I\mapsto J$ is a contraction mapping, and so by Banach's Fixed Point Theorem, 
the fixed point $x_*=f(x_*) \iff f^{-1}(x_*)=x_*$ is unique.   \myqed

\end{proof}

Uniqueness can also be shown in Cases (0,2), but with certain further restrictions on the parameter values.  For definiteness, 
the following theorem focuses on Case 0; a similar result holds for Case 2.

\begin{theorem}
Suppose that Case 0 of Equation~\eqref{eq:cases} holds with the restriction 
\[
\langle \psi \rangle > (f_0/v)+\sqrt{ 4(f_0/v)^2+\tfrac{1}{3}}.
\]
on the parameter values.
Then Equation~\eqref{eq:ch_reg1} has exactly 
one $L$-periodic solution.
\end{theorem}

\begin{proof}
As before, we look at $f'(\psi_0)$ on an appropriate sub-interval $J\subset I$:
\begin{equation}
f'(\psi_0)=\exp\left(\int_0^L \frac{\partial F}{\partial \psi}\big|_{\psi(\eta)}\mathd \eta\right),
\label{eq:formal1}
\end{equation}
where again the sign of $\partial F/\partial \psi$ determines the behaviour of the derivative of $f'(\psi_0)$.  As before,
\[
\frac{\partial F}{\partial \psi}=
\frac{3v\psi^2-6\psi \left[v\langle \psi\rangle-f_0\sin (k z)\right]+v}{(3\psi^2-1)^2}.
\]
In contrast to Case 1, in Case 0, the numerator of $\partial F/\partial\psi$ vanishes at $\psi$-values
\[
\psi_{\pm}=X\left(1\pm \sqrt{1 -\frac{1}{3X^2}}\right),\qquad
X=\langle \psi\rangle-(f_0/v)\sin(k z).
\]
We have $X>1/\sqrt{3}$ in Case 0, hence $\psi_+>1/\sqrt{3}$ and $\psi_-<1/\sqrt{3}$.  As such, only $\psi_+$ is admissible in Case 0.  The goal now is to show that $\psi(z)\neq \psi_+$ along a periodic trajectory in Case 0, which amounts
to showing
\[
\psi_+ \notin \left[\langle \psi\rangle-(f_0/v),\langle \psi \rangle + (f_0/v)\right],
\]
since the periodic trajectory is contained in this range.
As such, we view $\psi_+$ as a function of $z$ and we require 
\[
\min_{z}\psi_+(z)> \langle \psi\rangle + (f_0/v).
\]
We recall the definition of $\psi_+$ as a root of
\begin{equation}
3v\psi_+^2-6\psi_+ \left[v\langle \psi\rangle-f_0\sin (k z)\right]+v=0.
\label{eq:psiroot}
\end{equation}
By differentiating both sides of this equation with respect to $z$ and setting $\mathd \psi_+/\mathd z=0$ 
(corresponding to extreme values of $\psi_+(z)$), one obtains the condition $\psi_+(z)\cos (k z)=0$.  The possibility $\psi_+=0$ is ruled out in view of Equation~\eqref{eq:psiroot}, hence $\cos(kz)=0$ at the extreme points.
By inspection, the minimum is attained at $kz=\pi/2$, hence
\[
\min_{z}\psi_+(z) = 
\left[\langle \psi\rangle -(f_0/v)\right]\bigg\{1+ \sqrt{1 -\tfrac{1}{3}\frac{1}{\left[\langle \psi\rangle - (f_0/v)\right]^2}}\bigg\},
\]
and we therefore require
\[
\left[\langle \psi\rangle -(f_0/v)\right]\bigg\{1+ \sqrt{1 -\tfrac{1}{3}\frac{1}{\left[\langle \psi\rangle - (f_0/v)\right]^2}}\bigg\}
> \langle \psi\rangle+ (f_0/v).
\]
This simplifies to give
\begin{equation}
\langle \psi \rangle > (f_0/v)+\sqrt{ 4(f_0/v)^2+\tfrac{1}{3}},
\label{eq:range_thm}
\end{equation}
and this is precisely the range given in the theorem statement.

As such, provided $\langle \psi\rangle$ is in the range given by the inequality~\eqref{eq:range_thm}, $\left(\partial F/\partial\psi\right)_{\psi(\eta)}<0$, and hence $|f'(\psi_0)|<1$ for all $\psi_0\in J$ and hence, $f:J\rightarrow I$ is a contraction mapping and the known fixed point $x_*$ is unique.

\end{proof}

\begin{remark}
A similar result holds in Case 2: with the restriction 
\begin{equation}
\langle \psi \rangle < -(f_0/v)-\sqrt{ 4(f_0/v)^2+\tfrac{1}{3}},
\label{eq:range_thm1}
\end{equation}
the $L$-periodic solution is unique in that case.
\end{remark}
%

Finally, we note that although a unique base state is guaranteed only in the ranges given by
Equations~\eqref{eq:range_thm}--\eqref{eq:range_thm1}, we have been unable to find any numerical evidence for non-unique solutions in any other parts of the parameter space.
As such, the numerical analysis provided in Figure~\ref{fig:parameter_space_cj} (with at most one periodic solution at any point in parameter space) appears to be complete.
   
\section{The reduced-order model -- linear stability analysis}
\label{sec:stability}

We look at solutions of the temporally-evolving equation~\eqref{eq:ch_tw} that are made up of the equilibrium travelling-wave part $\psi(\eta)$ (the ``base state''), plus a small perturbation:
\begin{equation}
C(x,t)=\psi(\eta)+\delta C(\eta,t),\qquad \eta=x-vt,
\label{eq:trial}
\end{equation}
where $\psi(\eta)$ is the travelling-wave solution already characterized in Section~\ref{sec:reduced} and $\delta C$ is a small perturbation.  
We substitute Equation~\eqref{eq:trial} into Equation~\eqref{eq:ch_tw} and linearize, omitting terms that are $O(\delta C^2)$ and higher, 
and we obtain
\begin{equation}
\frac{\partial}{\partial  t}\delta C-v\frac{\partial}{\partial\eta}\delta C = 
D\frac{\partial^2}{\partial\eta^2}\left(\myspin \delta C\right)- \epsilon D \frac{\partial^4}{\partial\eta^4}\delta C,\qquad
\myspin =3\psi^2-1.
\label{eq:stability_gamma}
\end{equation}
This equation is separable: we can write $\delta C=\mathe^{\lambda t}\widetilde{\delta C}(\eta)$, such that Equation~\eqref{eq:stability_gamma} becomes (after omitting the tilde over $\widetilde{\delta C}(\eta)$:
\begin{equation}
\lambda \delta C-v\frac{\partial}{\partial\eta}\delta C = 
D\frac{\partial^2}{\partial\eta^2}\left(\myspin \delta C\right)
-\epsilon D \frac{\partial^4}{\partial\eta^4}\delta C.
\label{eq:stability_lambda}
\end{equation}
Equation~\eqref{eq:stability_lambda} is an eigenvalue equation in the eigenvalue $\lambda$: if $\myre(\lambda)>0$ for some eigenvalue in the spectrum of Equation~\eqref{eq:stability_lambda}, then the travelling wave $\psi(\eta)$ is unstable.  We continue working in the limit where the reduced-order model is valid; hence, we work with Equation~\eqref{eq:stability_lambda} with $\epsilon=0$.

{\textbf{As the full temporally-evolving Cahn--Hilliard equation with sinusoidal forcing preserves the mean concentration $\langle \psi\rangle$, it follows that $\langle \delta C\rangle=0$ for all time.  
Therefore the boundary conditions on $\delta C(x,t)$ are either (i) periodic, with $\delta C(\eta+L,t)=\delta C(\eta,t)$, or 
(ii) bounded, with $\delta C \rightarrow 0$ as $|\eta|\rightarrow \infty$ (and the same for the $\eta$-derivatives of $\delta C$) or
In this paper, we deal with Case (i) only, for the following reasons: this case is simple, and it can be used to shed light on the numerical simulations below in Section~\ref{sec:full}.  Also, the analysis developed in Case (i) may be combined with the theory developed in Reference~\cite{kenny2018}, such that Case (ii) may be considered an extension of Case (i).
As such, we focus in the rest of this section on periodic perturbations, with mean zero, specifically,
\begin{subequations}
\begin{equation}
\delta C(\eta+L,t)=\delta C(\eta,t),\qquad t>0.
\end{equation}
\begin{equation}
\langle \delta C\rangle(t)=0,\qquad t\geq 0.
\end{equation}%
\label{eq:bc_pert}%
\end{subequations}%
With these clarifying remarks, sufficient conditions for stability can be obtained in Regions 0 and 2:
\begin{theorem}
\label{thm:stable}
Travelling-wave solutions $\psi(\eta)$ in Regions 0 and 2 are stable 
with respect to mean-zero periodic perturbations if
\begin{equation}
 \myspin_{\mathrm{min}}(2\pi/L)^2 \geq \tfrac{1}{2}|\myspin''|_{\mathrm{max}}|.
\label{eq:constraint_ab}
\end{equation}
\end{theorem}
Here, $\myspin_{\mathrm{min}}>0$ is the minimum of $\myspin$ over the periodic interval $[0,L]$ and similarly,
\[
|\myspin''|_{\mathrm{max}}=\mathrm{max}_{[0,L]}|\myspin''|.
\]
\begin{proof}
We start with the eigenvalue problem~\eqref{eq:stability_lambda} with $\epsilon=0$.  We multiply both sides of Equation~\eqref{eq:stability_lambda} by $\delta C^*$ and integrate from $\eta=0$ to $\eta=L$, applying the periodic boundary conditions, to obtain the following relations:
\begin{eqnarray*}
\myre(\lambda)\|\delta C\|_2^2&=&
-\int_0^L \myspin 
\left|\frac{\mathd}{\mathd\eta}\delta C\right|^2\mathd \eta
+\tfrac{1}{2}\int_0^L \myspin''|\delta C|^2\,\mathd \eta,\\
&\leq & -\myspin_{\mathrm{min}}
\int_0^L 
\left|\frac{\mathd}{\mathd\eta}\delta C\right|^2\mathd \eta
+\tfrac{1}{2}|\myspin''|_{\mathrm{max}}\int_0^L |\delta C|^2\,\mathd \eta.
\end{eqnarray*}
In view of Equation~\eqref{eq:bc_pert}, the perturbation $\delta C$ has mean zero, hence Poincar\'e's inequality applies to the above string of relations, and we thereby obtain
\[
\myre(\lambda)\|\delta C\|_2^2\leq
-\myspin_{\mathrm{min}}(2\pi/L)^2\|\delta C\|_2^2+\tfrac{1}{2}|\myspin''|_{\mathrm{max}}\|\delta C\|_2^2.
\]
Therefore, if
\begin{equation}
 \myspin_{\mathrm{min}}(2\pi/L)^2 \geq \tfrac{1}{2}|\myspin''|_{\mathrm{max}}|.
\label{eq:reg02}
\end{equation}
then it follows that $\myre(\lambda)\leq 0$.  \myqed
\end{proof}
A similar result follows for Region 1, where it can be shown that the mean-zero periodic perturbations are unstable if
\begin{equation}
|\myspin|_{\mathrm{min}}(2\pi/L)^2\geq \tfrac{1}{2}|\myspin''|_{\mathrm{max}}|,
\qquad
|\myspin|_{\mathrm{min}}=\mathrm{min}_{[0,L]}|\myspin|.
\label{eq:reg1}
\end{equation}
%
%
The regions of stability and instability mapped out by Equations~\eqref{eq:reg02}--\eqref{eq:reg1} depend \textit{a priori} on the the basic concentration profile $\psi(\eta)$, which can be generated only numerically.  As such, we have computed $\psi(\eta)$ numerically for a range of values in the parameter subspace $\left(\langle \psi\rangle, f_0\right)$ for the case $v=1$, in Figure~\ref{fig:fullscan}.  From this figure, it can be seen that  in a large part of Region 0, the basic concentration profile is stable , while there is a small shaded part of Region 0 where the basic profile may be unstable or stable -- with furter information required.  A similar picture holds in Region 1.  To complete the picture,  in Section~\ref{sec:full} below we resort to numerical calculations to map out the stable and unstable regions of parameter space more comprehensively. 
}

\begin{figure}
\centering
	\includegraphics[width=0.6\textwidth]{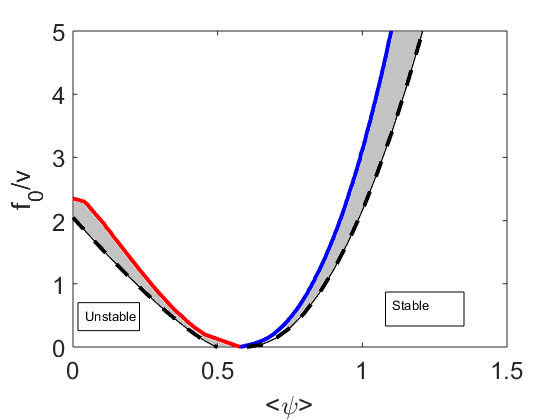}
	\caption{Reduced-order model: Plot of the parameter subspace $\left(\langle \psi\rangle,f_0\right)$ at fixed $v=1$. A large part of Region 0 corresponds to stable travelling waves as indicated.  In the remaining part of Region 0 (shaded), the stability of the travelling waves is not known \textit{a priori}.  A similar picture holds in Region 1.
}
\label{fig:fullscan}
\end{figure}


These results can be understood in the context of the classical spinodal instability of the Cahn--Hilliard equation without forcing (i.e. Equation~\eqref{eq:ch_tw2} with $f_0=0$).  In this equation, a constant state $c_0$ is linearly unstable when $3c_0^2-1<0$.  This corresponds to spinodal instability, which is the mechanism that drives phase separation in (unforced) binary mixtures~\cite{CH_orig}.  In contrast, for highly asymmetric mixtures (i.e. $c_0\neq 0$, with $3c_0^2-1>0$), the spinodal instability is suppressed.  Our results for the travelling-waves in the forced Cahn--Hilliard equation can therefore be viewed as an extension of this classical instability.
\section{The full model equation}
\label{sec:full}

We now look in detail at parameter cases where the reduced-order model breaks down, such that a solution of the full model is required, recalled here as
\begin{equation}
\epsilon D \frac{\mathd^3\psi}{\mathd\eta^3}=
D\frac{\mathd}{\mathd\eta}\left(\psi^3-\psi\right)+v\left(\psi-\langle \psi\rangle\right)+f_0\sin(k\eta).
\label{eq:ch_tw2x}
\end{equation}
The starting-point of the study is numerical simulation of the temporally-evolving counterpart of Equation~\eqref{eq:ch_tw2x}, i.e. temporally-evolving numerical simulations (TENS), based on Equation~\eqref{eq:ch_tw}.  Results of these simulations are reported in what follows.   All simulations are performed at fixed $\epsilon=5\times 10^{-4}$, and for various values of the parameters $\langle \psi\rangle$ and $f_0$.  The parameter $v$ is fixed as $v=1$, although the effect of varying $v$ is investigated briefly below in Section~\ref{sec:vary_v}.

\subsection{Overview of Results}

An overview of the results of the TENS is given in 
Figure~\ref{fig:flow_map}.
In various parts of the parameter space, the TENS lead to a steady travelling-wave profile, whereas in other parts of the parameter space, no such steady state exists.  The steady profiles correspond to solutions of
Equation~\eqref{eq:ch_tw2x}.
\begin{figure}
	\centering
		\includegraphics[width=0.7\textwidth]{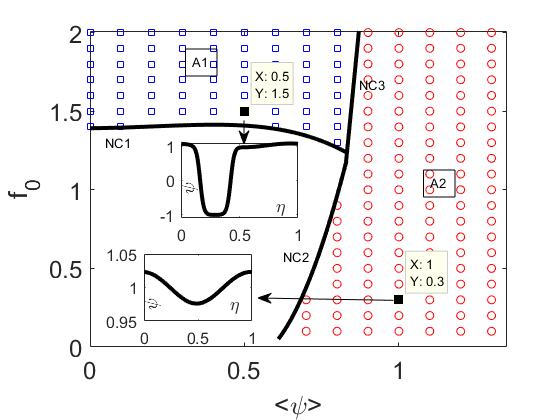}
	\caption{Summary of results of temporally-evolving numerical simulations for fixed $v=1$,
	and for various values of $\langle \psi\rangle$ and $f_0$.  Also, the small parameter $\epsilon$ is set to $5\times 10^{-4}$.  The circles and squares indicate simulations where a steady travelling wave exists.  }
\label{fig:flow_map}
\end{figure}
Two distinct steady profiles manifest themselves in the figure, in distinct parts of the parameter space.  The first profile (labelled `A2' in the figure) consists essentially of an oscillation around the mean value $\langle \psi\rangle$.  The second profile (labelled `A1') consists of regions wherein $\psi\approx 1$ and $\psi \approx -1$, joined together across transition regions, such that $L^{-1}\int \psi(\eta)\mathd \eta=\langle \psi\rangle$.  These distinct profile types have been identified previously in Reference~\cite{weith2009traveling} -- but only in the case $\langle \psi\rangle=0$ (the notation for A1 and A2 is the same as that used in the earlier work).

The steady concentration profiles in Figure~\ref{fig:flow_map} indicate the existence of steady-state solutions of the traveling-wave equation~\eqref{eq:ch_tw2x}.  
These are herein constructed in an equivalent yet independent fashion using a Newton solver (see Section~\ref{sec:problem} and~\ref{app:methodology}); the solutions computed in this manner coincide exactly with the results of the TENS.  This second independent approach is useful because it forms the basis of a linear stability analysis.  As such, we take the steady-state profiles computed via the Newton solver and substitute them into the full linear-stability equation~\eqref{eq:Phi_p}.  Using this analysis, the neutral curves in Figure~\ref{fig:flow_map} are generated -- the neutral curves give the precise limits of the regions in parameter space where the TENS yield steady-state travelling-wave solutions.

Using the numerically-generated travelling-wave solutions (i.e. those generated with the Newton solver), we can characterize the region in 
Figure~\ref{fig:flow_map} where no travelling waves are found via the TENS.  As such, by crossing the curve NC1, the mode A1 continues to exist (as confirmed by the solutions generated with the Newton solver), although the mode A1 switches from stable (high values of $f_0$) to unstable (lower values of $f_0$), across the curve NC1.  Similarly, by crossing NC2, the mode A3 loses stability.  Finally, along the neutral curve NC3, the modes A1 and A2 undergo an exchange of stability (as evidenced by Figure~\ref{fig:mode_A1_changeover}).
\begin{figure}
	\centering
		\subfigure[$\,\,$A1]{\includegraphics[width=0.45\textwidth]{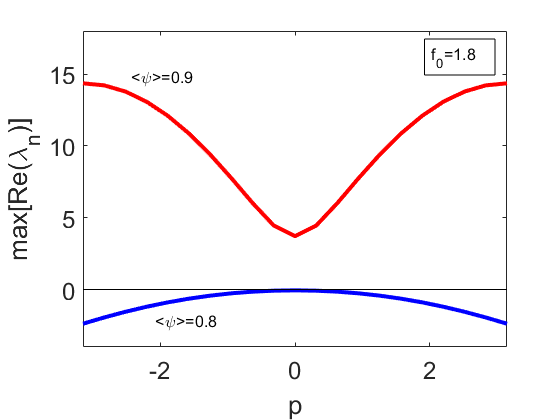}}
		\subfigure[$\,\,$A2]{\includegraphics[width=0.45\textwidth]{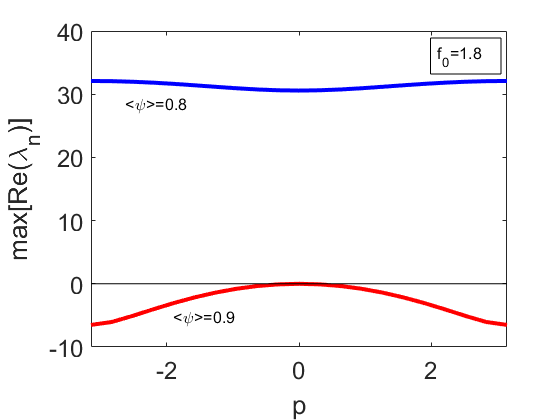}}
		\caption{Exchange of stability between solutions A1 and A2 as the neutral curve N3 is traversed.}
	\label{fig:mode_A1_changeover}
\end{figure}

\subsection{Discussion}

The different solutions in Figure~\ref{fig:flow_map} have been constructed using the Newton solver by providing that method with a specific initial guess for the solution.  Several initial guesses have been provided, leading to the two flow profiles observed in Figure~\ref{fig:flow_map}, as well as other profiles, which we outline below.  

\paragraph*{The A2-solution}  A2-type solutions are observed in Figure~\ref{fig:flow_map}.  These are similar to the solutions of the reduced-order model, which in turn can be identified with an oscillatory profile which oscillates around the mean profile $\langle \psi\rangle$.  The oscillation is either linear or nonlinear.  In the linear case, the oscillation has the same single characteristic wavenumber as the forcing; this is the scenario explained by the linearized solutions of the reduced-order model identified in Section~\ref{sec:reduced}.  Otherwise, the oscillation is nonlinear, and is characterized by the fundamental wavelength $k=2\pi$, and by higher harmonics, which give rise to a steepened concentration profile.  In both scenarios, a solution of of the linearized (full) model corresponding to an oscillation around the mean profile is used as an initial guess for the Newton solver, which leads to the A2-type solutions in the parts of the parameter space outlined in Figure~\ref{fig:flow_map}.

\paragraph*{The A1-solution}  A1-type solutions are also observed in Figure~\ref{fig:flow_map}.  These can be understood intuitively, with Equation~\eqref{eq:ch_tw2x} as the starting point.  As $\epsilon\rightarrow 0$, the spatial variations in Equation~\eqref{eq:ch_tw2x} separate into rapid variations on the scale $\epsilon^{1/2}$, and slow variations on the scale $L$.  If we furthermore look at the limiting case with $v\rightarrow 0$ (or $f_0\rightarrow \infty$), the slow variations are governed by the balance
\[
D\frac{\mathd}{\mathd\eta}\left(\psi^3-\psi\right)\sim f_0\sin(k\eta),
\]
hence
\begin{equation}
\psi^3-\psi\sim -[f_0/(Dk)]\cos (k\eta)+\beta,
\label{eq:balance}
\end{equation}
where $\beta$ is a constant of integration.
This is a cubic equation in $\psi$, with at most three real solutions labelled as $f_j(\eta)$. 
  As such, the A1 solution consists of a patchwork of two such $f_j$ functions, stitched together by two transition regions of width $\epsilon^{1/2}$.  A possibly infinite family of such solutions exists, parameterized by $\beta$.  However, a particular solution is selected such that $\max(\psi)\approx 1$ and $\min(\psi)\approx -1$, as these values are energetically the most favourable in the free-energy picture of the Cahn--Hilliard dynamics (\textit{cf.} Equation~\eqref{eq:ch_basic}).  A sketch of this idea is shown in Figure~\ref{fig:my_crazy_profile}.  
\begin{figure}
	\centering
		\includegraphics[width=0.6\textwidth]{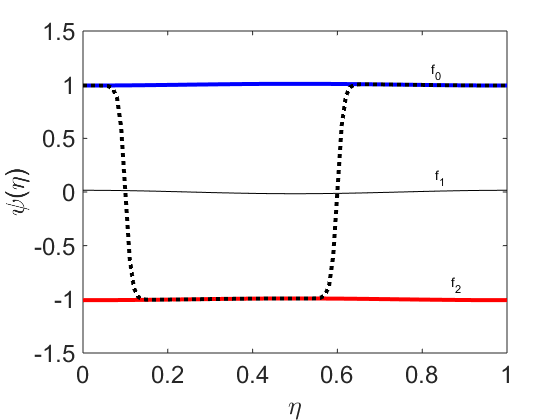}
		\caption{The idea for the construction of the single-spike approximate solution to Equation~\eqref{eq:ch_tw2x}.  The approximate solution is constructed by stitching together two $f_j$-profiles.  Recall, the $f_j$-functions are solution of Equation~\eqref{eq:balance}.  In this figure,  the $f_j$ profiles are joined together across narrow step-like transition regions.  }
	\label{fig:my_crazy_profile}
\end{figure}

This approach is similar to the idea of matched asymptotic expansions.  Although we do not use a rigorous theory of matched asymptotic expansions here, the terminology of that theory is useful.  As such, the $f_j$-profiles can be regarded as `outer solutions'.
Furthermore, in the limit as $\epsilon\rightarrow 0$, the spatial variation in $f_j(\eta)$ is negligible, as all spatial variations take place over the transition regions.  Then, the outer solutions are determined by $\psi^3-\psi=\beta$, with $\psi=\pm 1,0$ for $\beta=0$.  
Correspondingly, there is an `inner problem', where the dominant balance in Equation~\eqref{eq:ch_tw2x} is given by 
\[
\epsilon  \frac{\mathd^3\psi}{\mathd\eta^3}\sim 
\frac{\mathd}{\mathd\eta}\left(\psi^3-\psi\right),
\]
solutions of which are $\tanh$-functions (\textit{cf.} Equation~\eqref{eq:mytanh}).
By combining the inner and outer solutions in a heuristic fashion, an approximate single-spike solution can be constructed as
\begin{equation}
\psi^{\mathrm{approx}}(\eta)=s \tanh\left( \frac{\eta-c_1}{\sqrt{2\epsilon}}\right)\tanh\left( \frac{\eta-c_2}{\sqrt{2\epsilon}}\right),\qquad \epsilon\rightarrow 0,\qquad s=1,
\label{eq:psi_approx}
\end{equation}
where $c_1=L/4$ and $c_2=c_1+\tfrac{1}{2}\left[L-s\langle \psi\rangle\right]$.  This choice of $c_1$ and $c_2$ has the effect of stitching together the outer solutions such that $\langle \psi^{\mathrm{approx}}\rangle$ has the required value, 
$\langle \psi^{\mathrm{approx}}\rangle=\langle \psi\rangle$.  
Finally, a finite value of $v$ can be introduced to this theory: the effect of $v$ is  to introduce a phase shift in the $\psi$-profile so constructed, relative to the $v=0$ case.

These intuitive arguments are the basis for using $\psi^{\mathrm{approx}}(\eta)$ in Equation~\eqref{eq:psi_approx} as an initial guess for the Newton solver.  The results of iterating the Newton solver confirm the correctness of the theory, as the solver converges to a concentration profile almost identical (up to a phase shift) to the initial guess $\psi^{\mathrm{approx}}(\eta)$ (e.g. Figure~\ref{fig:close_initial}).
\begin{figure}
	\centering
		\includegraphics[width=0.6\textwidth]{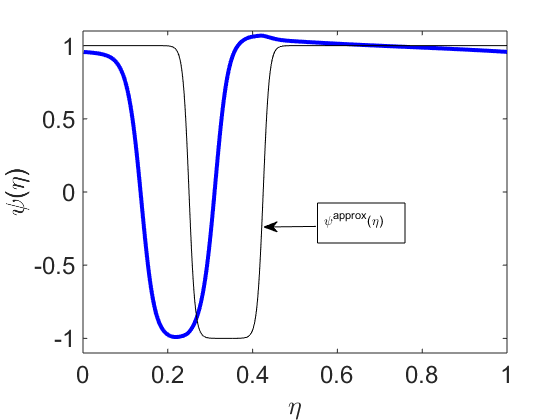}
		\caption{Emergence of single-spike concentration profile from the inital guess $\psi^{\mathrm{approx}}(\eta)$ given by Equation~\eqref{eq:psi_approx}.  Parameters: $\langle \psi\rangle=0.65$, $f_0=0.1$.}
	\label{fig:close_initial}
\end{figure}

We emphasize that this approach can be placed on a much more solid footing when $v=0$.  This special case has been studied extensively in the context of resonant sloshing in shallow water waves~\cite{ockendon1986resonant,mackey2003dynamics}.  Using these works, Equation~\eqref{eq:ch_tw} with $v=0$ can be shown to reduce to a forced Duffing Oscillator, which possesses a Hamiltonian structure. 
Perturbation methods based on this Hamiltonian structure~\cite{mackey2003dynamics} and more generic methods~\cite{ockendon1986resonant} can then be used can be used to construct infinitely-many steady-state solutions which occur as $\epsilon\rightarrow 0$, and which exhibit spatial chaos.  As such,  the comparison with the $v=0$ case, also gives a motivation to seek out further steady-state concentration profiles in what follows.

\paragraph*{The A3 solution} A counterpart to the A1-solution is found by taking $s=-1$ in Equation~\eqref{eq:psi_approx}, and applying this as an initial guess to the Newton solver.  This yields another single-spike solution (not shown).  This is consistent with the fact that the two outer solutions $\psi\approx \pm 1$ can be stitched together in two distinct ways to produce the full solution.  The A3 solution has already been found by different means for the case $\langle \psi\rangle=0$ and was found there to be linearly unstable.  Since the A3 solution is not observed in any of the TENS in Figure~\ref{fig:flow_map}, it can be concluded that the A3 solution is unstable for general values of $\langle \psi \rangle$.

\paragraph*{Other solutions}  Motivated by the above considerations, we have initialized the Newton-solver with a `multiple-spike' initial condition.  The aim here is to demonstrate that the full model possesses such multiple-spike travelling-wave solutions.  As such, we have initialized the Newton solver with the following $N$-spike initial solution guess:
\begin{equation}
\psi^{\mathrm{approx}}(\eta)=(\pm 1)\prod_{j=0}^N \tanh\left( \frac{N\eta-j-c_1}{\sqrt{2\epsilon}}\right)\tanh\left( \frac{N\eta-j-c_2}{\sqrt{2\epsilon}}\right).
\label{eq:cond1}
\end{equation}
where as before, $c_1=L/4$ and $c_2=c_1+\tfrac{1}{2}\left[L-(\pm 1)\langle \psi\rangle\right]$; this provides for  $\langle\psi^{\mathrm{approx}}\rangle= \langle \psi\rangle$.   
The results are shown in Figures~\ref{fig:two_spike}--\ref{fig:mult_spike} for selected parameter values;  these plots establish the existence of multiple-spike travelling-wave solutions.   Two-spike solutions are found at $\langle \psi\rangle=0.65$, whereas $N$-spike solutions are found at $\langle \psi\rangle=0.1$, with $N=2,3,4,5,6$.  After $N=6$, the Newton solver fails to pick out multiple-spike solutions and the solver converges to an A2-type solution.  
%
\begin{figure}
	\centering
		\subfigure[$\,\,\langle \psi\rangle=0.65$]{\includegraphics[width=0.45\textwidth]{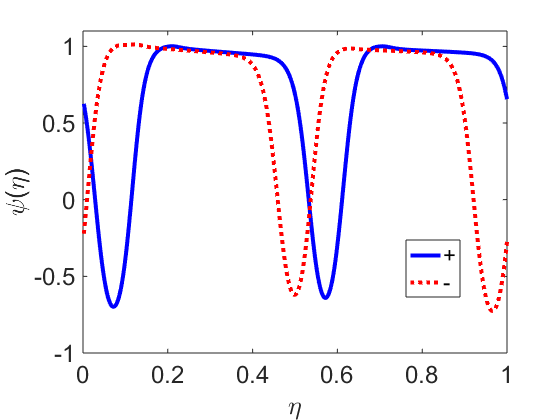}}
		\caption{Two-spike solutions for $(f_0,v)=(0.1,1)$ and $\langle \psi\rangle=0.65$.
		This parameter choice corresponds to a region in parameter space where only the A2 travelling wave is linearly stable.}
\label{fig:two_spike}
\end{figure}
\begin{figure}
		\subfigure[$\,\,N=2$]{\includegraphics[width=0.32\textwidth]{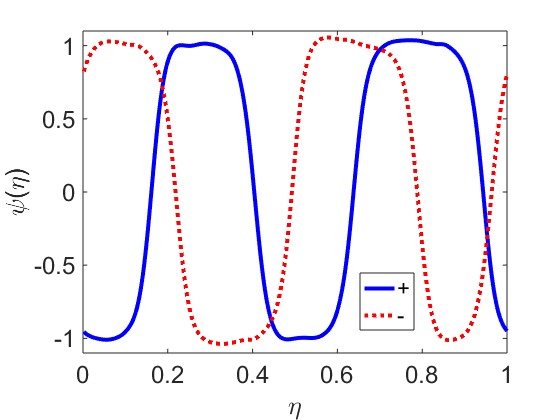}}
		\subfigure[$\,\,N=3$]{\includegraphics[width=0.32\textwidth]{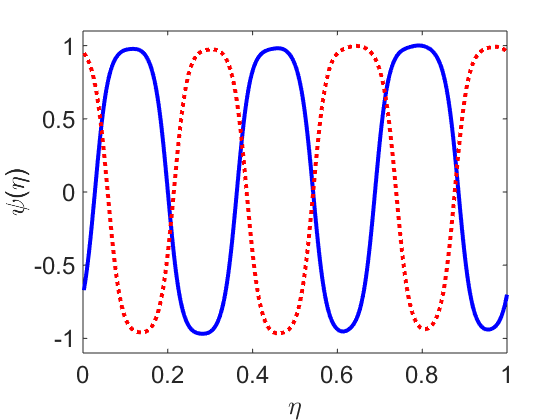}}
		\subfigure[$\,\,N=6$]{\includegraphics[width=0.32\textwidth]{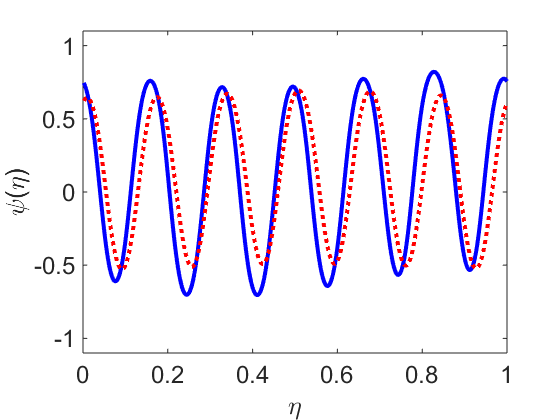}}
		\caption{Multiple-spike solutions for $(f_0,v)=(0.1,1)$ and $\langle \psi\rangle=0.1$, 
		corresponding to a region in parameter space with no stable travelling waves.}
	\label{fig:mult_spike}
\end{figure}
However, by reducing $\epsilon$ further, more spiked solutions are recovered (e.g. $N=15$ with $\epsilon=10^{-4}$, not shown).
It can also be checked via linear stability analysis that these solutions are  unstable (e.g. $\max[\myre(\lambda)]\approx 1669$ with $(f_0,v)=(0.1,1)$, $\langle \psi\rangle=0.1$, and $\epsilon=10^{-4}$ for a 15-spike solution).




\subsection{The effect of variations in $v$}
\label{sec:vary_v}

In Remark~\ref{rem:params} we identified $\langle \psi\rangle$, $f_0$, and $v$ as the key parameters controlling the steady-state concentration profiles (the positive parameter $\epsilon$ matters as well, although this is assumed to be small, such that its exact value is not directly relevant).  So far we have highlighted the effect of varying $\langle \psi\rangle$ and $f_0$ on the steady-state concentration profile -- see e.g. Figure~\ref{fig:flow_map}.  We therefore complete the parametric study by studying the effect of variations of $v$ on the concentration profiles.  As such, in Figure~\ref{fig:flow_map1} we summarize the results of further TENS for various values of $v$.  Figure~\ref{fig:flow_map1} may be compared with Figure~\ref{fig:flow_map}, in which $v=1$.
\begin{figure}
	\centering
		\subfigure[$\,\,v=0.5$]{\includegraphics[width=0.45\textwidth]{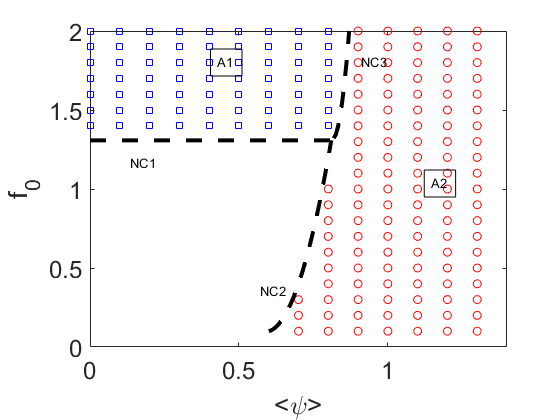}}
		\subfigure[$\,\,v=2$]{\includegraphics[width=0.45\textwidth]{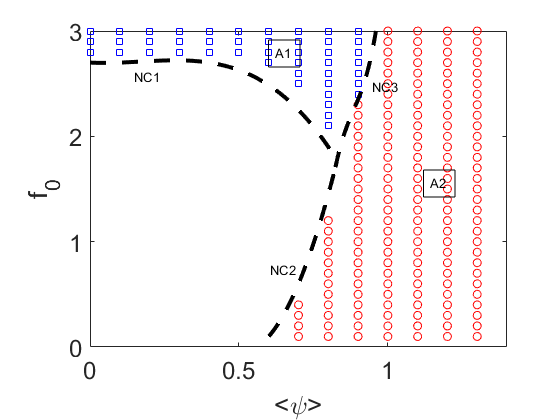}}
		\caption{Summary of results of temporally-evolving simulations for various values of the parameters $v$, $\langle \psi\rangle$,  $f_0$.  The small parameter $\epsilon$ is set to $5\times 10^{-4}$.  As in Figure~\ref{fig:flow_map}, the circles and squares indicate simulations where a steady travelling wave exists.}
	\label{fig:flow_map1}
\end{figure}
The dominant effect to be seen in Figure~\ref{fig:flow_map1} is that increasing $v$ is destabilizing: the neutral curve NC1 shifts to higher $f_0$-values with increasing $v$.  The upward shift in NC1 therefore corresponds to an increase in the blank region in the flow-pattern map where no stable steady-state travelling-waves exist.  We remark that the A1-solution persists in this region, but it is linearly unstable, as in Figure~\ref{fig:flow_map}.  Changing $v$ also causes the curves NC2 and NC3 to shift slightly, but these are small effects compared to the shift in NC1.

\section{Conclusions}

Summarizing, we have revisited the problem of the one-dimensional forced Cahn--Hilliard equation with travelling-wave forcing, as a model for phase separation.  Steady-state travelling wave solutions emerge from the model equation, as evidenced by transient numerical simulations.  In contrast to the earlier studies (e.g. References~\cite{weith2009traveling,naraigh2008bounds}), we look at the mean concentration level as a key parameter.  This enables us to characterize the travelling-wave solutions in depth -- using both analytical and numerical techniques.  In the limiting case where the phase-separation scale $\epsilon$ tends to zero, we have identified certain regions of parameter space where a highly simplified, reduced-order Cahn--Hilliard model pertains; in the other regions, the full Cahn--Hilliard equation is required.  In these other regions, we have used a type of singular perturbation theory to demonstrate the existence of a whole zoo of multiple-spiked solutions, all of which are unstable.  These are interesting from a mathematical point of view, as they can be related to solutions of a forced Duffing Oscillator; they   may be of further interest in characterizing the transient dynamics of the forced Cahn--Hilliard phase separation, e.g. in those further regions of parameter space where no steady-state stable travelling-wave solutions exist.

\subsection*{Acknowledgements}

LON thanks the staff of the Bray Institute for Advanced Studies for their hospitality during the course of this research.  LON also acknowledges insightful discussions with Ted Cox.

\appendix

\section{Detailed description of the numerical methodologies}
\label{app:methodology}

We develop in detail the various numerical methods used to Equation~\eqref{eq:ch_tw} and its steady-state travelling-wave counterparts~\eqref{eq:ch_tw2} and~\eqref{eq:ch_reg}.  We refer frequently to these equations throughout this appendix, it is therefore helpful to recall them all together here as follows (notation as in the main paper, with $\gamma$ replaced by its dimensionless equivalent $\epsilon$): 
\begin{itemize}
\item Reduced-order model with steady-state travelling-wave solutions (i.e. Equation~\eqref{eq:ch_reg}):
\begin{equation}
0=
D\frac{\mathd}{\mathd\eta}\left(\psi_0^3-\psi_0\right)+v\left(\psi_0-\langle \psi\rangle\right)+f_0\sin(k\eta).
\label{eq:app:ch_reg}
\end{equation}
\item Full model with steady-state travelling-wave solutions (i.e. Equation~\eqref{eq:ch_tw2}):
\begin{equation}
\epsilon D \frac{\mathd^3\psi}{\mathd\eta^3}=
D\frac{\mathd}{\mathd\eta}\left(\psi^3-\psi\right)+v\left(\psi-\langle \psi\rangle\right)+f_0\sin(k\eta);
\label{eq:app:ch_tw2}
\end{equation}
\item Temporally-evolving equation with travelling-wave source (i.e. Equation~\eqref{eq:ch_tw}):
\begin{equation}
\frac{\partial C}{\partial t}=D\partial_{xx}\left(C^3-C-\epsilon\partial_{xx} C\right)+f_0k\cos[k(x-vt)];
\label{eq:app:ch_tw}
\end{equation}
\end{itemize}
Each of these ordinary differential equations (ODEs) is solved with $L$-periodic boundary conditions in either the variable $\eta$ (Equations~\eqref{eq:app:ch_reg},\eqref{eq:app:ch_tw2}), or the variable $x$ (Equation~\eqref{eq:app:ch_tw}).

\subsection{Steady-state travelling-wave solutions}

We begin by solving Equation~\eqref{eq:app:ch_reg}.  The ODE is solved in a straightforward fashion using an eighth-order accurate Runge-Kutta scheme~\cite{ode87}. The periodic boundary conditions are imposed numerically using a `shooting' method: a variable boundary condition $\psi(\eta=0)=a$ is imposed.  The solution of the ODE then produces an $a$-dependent concentration profile $\psi(\eta;a)$.  The value of $a$ is then adjusted such that the periodic boundary conditions hold, i.e.
\begin{equation}
\psi(L,a)=a.
\label{eq:app:root}
\end{equation}
Thus, the problem of enforcing the boundary condition on the ODE~\eqref{eq:app:ch_reg} is reduced to rootfinding (i.e. Equation~\eqref{eq:app:root}); this can be achieved using standard numerical techniques.

A second independent method is introduced to solve Equation~\eqref{eq:app:ch_reg}.  This is used in the main text to confirm results; this second method also carries over very straightforwardly to the full model problem~\eqref{eq:app:ch_tw}.  As such, the solution of Equation~\eqref{eq:app:ch_reg} is discretized at $N$ equally-spaced points $\eta_i=i(L/N)$, with $i\in \{1,2,\cdots,N\}$ and $\Delta \eta=\eta_2-\eta_1$.  As such, a numerical solution is generated with value $\psi_i$ at the corresponding point $\eta_i$.  Furthermore, an $O(\Delta \eta^4)$-accurate to $\mathd\psi/\mathd \eta$ is introduced using a finite-difference method:
\begin{equation}
\left(\frac{\mathd\psi}{\mathd\eta}\right)_{\eta=\eta_i}=
\frac{\tfrac{1}{12}\psi_{i-2}-\tfrac{2}{3}\psi_{i-1}+\tfrac{2}{3}\psi_{i+1}-\tfrac{1}{12}\psi_{i+2}}{\Delta \eta^2}+O(\Delta \eta^4),\qquad i=3,4,\cdots,N-2.
\label{eq:app:d1}
\end{equation}
Suitable modifications are made to Equation~\eqref{eq:app:d1} near the boundary at $i=1,2,N-1,N$ to account for the periodic boundary conditions.  Equation~\eqref{eq:app:d1} defines a differentiation operator (matrix) $\mydiff$: if $\mypsivec=(\psi_1,\cdots,\psi_N)^T$, then the derivative vector is defined in an obvious way as $\mydiff\mypsivec$.  As such, a discretized version of Equation~\eqref{eq:app:ch_reg} is developed:
\begin{equation}
\myFvec(\mypsivec)=D \mydiff\left[\mypsivec\bullet\mypsivec\bullet\mypsivec-\mypsivec\right]+v\left[\mypsivec-
\langle \psi\rangle\myonevec\right]+\myfvec,
\qquad
\myFvec(\mypsivec)=0,
\label{eq:app:newton}
\end{equation}
where $\myonevec=(1,\cdots,1)^T$, $\myfvec=f_0\left(\sin(k\eta_1),\cdots,\sin(k\eta_N)\right)^T$, and the $\bullet$ denotes pointwise multiplication of vectors.

Equation~\eqref{eq:app:newton} is a set of $N$ nonlinear algebraic equations, i.e. $\myFvec(\mypsivec)=0$.  These are solved using a Newton-type algorithm, which we outline as follows.  We note first of all that the solution is contained in a high-dimensional space $\mathbb{R}^N$.  An initial guess $\mypsivecn$ for the solution is prescribed:
\begin{equation}
\mypsivecn=(\psi_1^n,\cdots,\psi_N^n).
\end{equation}
From this, a new guess $\mypsivecnpone$ is constructed by moving away from the initial guess, in a particular direction in $\mathbb{R}^N$.  The direction is given by the Jacobian of the nonlinear equations~\eqref{eq:app:newton}, specifically,
\begin{equation}
\myjac=D\mydiff\myS+v\mathbb{I}_{N\times N},
\end{equation}
where $\myS$ is a diagonal matrix with entries
\begin{equation}
(\myS)_{ii}=3(\psi_i^n)^2-1
\end{equation}
As such, the updated guess is given by
\begin{equation}
\mypsivecnpone=\mypsivecn+\delta\mypsivec,\qquad
\delta\mypsivec=-\myjac^{-1}\left[\myFvec(\mypsivecn)\right].
\label{eq:app:newt_alg}
\end{equation}
Equation~\eqref{eq:app:newt_alg} is the standard Newton's method for solving the system $\myFvec(\mypsivec)=0$.  The iterative process 
in~\eqref{eq:app:newt_alg} is continued  until the residual
\[
f[\mypsivecn]=\tfrac{1}{2}[\myFvec(\mypsivecn)]\cdot[\myFvec(\mypsivecn)]
\]
is zero, to within a small tolerance.

In practice, Equation~\eqref{eq:app:newton} appears to have a unique solution, corresponding to the (apparently) unique solution of the original boundary-value problem~\eqref{eq:app:ch_reg}
However, we will extend the method to other scenarios where several distinct solutions definitely exist.  In such a scenario, the initial guess may be close to a number of solutions, and control may be lost over the solution to which the algorithm converges.  To regain  control, we modify the basic Newton's method to add backtracking line search functionality~\cite{press1996numerical}.  As such, we introduce
\begin{equation}
f_0=f[\mypsivecn],\qquad f_k=f[\mypsivecnpone].
\label{eq:fk0}
\end{equation}
Let $\alpha_k=1$.  If
\[
f_k>f_0+c\alpha_k \left[\myFvec(\mypsivecn)\cdot\delta\mypsivec\right],
\qquad c=\mathrm{Const.}
\]
we reject the updated guess $\mypsivecnpone=\mypsivecn+\delta\mypsivec$.  We reduce $\alpha_k$ by letting $\alpha_k\rightarrow \rho\alpha_k$ (with $\rho<1$) and we compute a revised updated guess
\begin{equation}
\mypsivecnpone=\mypsivecn+\alpha_k\left(\delta\mypsivec\right).
\label{eq:fk}
\end{equation}
We also recompute $f_k$ using Equations~\eqref{eq:fk0}--\eqref{eq:fk}.  This defines a second iterative process, i.e. an inner iterative process, which continues until  $f_k\leq f_0$.  At the termination of each round of the inner iterative process, we continue with the next step of the outer iterative process over the iteration variable $n$.  The entire set of nested iterative processes is continued until $f[\mypsivecn]$ is zero, to within a tolerance.  The values of $c$ and $\rho$ are picked by trial and error, with reference to standard 
practice~\cite{linesearch}: specifically, we take $c=10^{-4}$ and $\rho=0.5$.

A sample result showing a comparison between $\psi$-profiles generated by the `shooting' method and the Newton solver is shown in Figure~\ref{fig:validate}.  The exact agreement between the methods demonstrated in this figure provides evidence that the two independent numerical methods have been implemented correctly.
\begin{figure}
	\centering
		\includegraphics[width=0.6\textwidth]{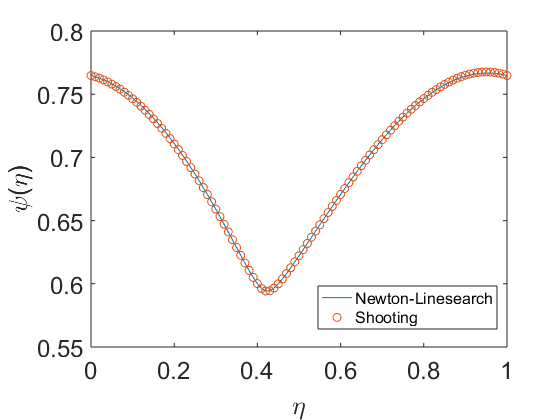}
		\caption{Comparison of the $\psi$-profiles for the two different solution methods: `shooting' method and Newton solver with backtracking line search.  Parameter values: $\langle \psi\rangle=0.7$, $f_0=1$, $v=1$.  In the Newton solver, $N=500$ gridpoints have been used in the spatial discretization.  The  residual $\|\mypsivecnpone-\mypsivecn\|_2/N$ is zero after 30 iterations of the solver.}
	\label{fig:validate}
\end{figure}
We emphasize that the above method (Newton solver with backtracking line search) can be extended in a straightforward fashion to the full model by adding a discretized third-order derivative to Equation~\eqref{eq:app:newton}.  This modification is so straightforward that no further explanation is required.

\subsection{Temporally-evolving solutions}

We start with Equation~\eqref{eq:app:ch_tw}, which we rewrite in the moving frame $\eta=x-vt$ as
\begin{equation}
\frac{\partial C}{\partial t}-v\frac{\partial C}{\partial \eta}=D\partial_{\eta\eta}\left(C^3-C-\epsilon\partial_{\eta\eta} C\right)+f_0k\cos(k\eta),\qquad
C(\eta+L,t)=C(\eta).
\label{eq:app:ch_tw_eta}
\end{equation}
Because of the periodic boundary conditions in Equation~\eqref{eq:app:ch_tw_eta}, we can expand the solution $C(\eta,t)$ in a Fourier series,
\[
C(\eta,t)=\sum_{\myk=-\infty}^\infty \mychat_{\myk}(t)\mathe^{\imag(2\pi/L)\myk\eta},
\]
where
\[
\mychat_{\myk}(t)=\frac{1}{L}\int_0^L \mathe^{-\imag(2\pi/L)\myk\eta}C(\eta,t)\mathd\eta.
\]
We further multiply Equation~\eqref{eq:app:ch_tw_eta} by 
$\mathe^{\imag(2\pi/L)\myk\eta}$ and integrate over $[0,L]$.  We obtain
\begin{equation}
\frac{\mathd \mychat_{\myk}}{\mathd t}=
-\epsilon D (2\pi/L)^4\myk^4\mychat_{\myk}
- D(2\pi/L)^2\myk^2\int_0^L \mathe^{\imag (2\pi/L)\myk \eta}(u^3-u)\mathd \eta
+\tfrac{1}{2}(f_0 k)\left(\delta_{\myk,1}+\delta_{\myk,-1}\right).
\label{eq:fkppperiodic1}
\end{equation}
We introduce
\[
Q=u^3-u,\qquad \widehat{Q}_{\myk}=\int_0^L \mathe^{\imag (2\pi/L)\myk \eta}(u^3-u)\mathd \eta
\]
Assuming perfect knowledge of all the Fourier coefficients of $u$ and $Q$, Equation~\eqref{eq:fkppperiodic1} can be discretized in time by definining a solution at discrete time points
\[
\mychat_{\myk}^{n}=\mychat_{\myk}(t=n\Delta t),\qquad n\in\{0,1,2,\cdots\},
\]
where $\Delta t$ is the timestep.  For numerical stability~\cite{Zhu_numerics}, Equation~\eqref{eq:fkppperiodic1} is discretized using a backward-Euler scheme:
\[
\frac{\mychat_{\myk}^{n+1}-\mychat_{\myk}^{n}}{\Delta t}=
-\epsilon D (2\pi/L)^4 \myk^4 \mychat_{\myk}^{n+1}  - (2\pi/L)^2\myk^2 \widehat{Q}_{\myk}^n+\tfrac{1}{2}(f_0 k)\left(\delta_{\myk,1}+\delta_{\myk,-1}\right).
\]
hence
\begin{equation}
\mychat_{\myk}^{n+1}=\frac{\mychat_{\myk}^{n}- \Delta t\,\widehat{Q}_{\myk }^n + \tfrac{1}{2}\Delta t(f_0 k)\left(\delta_{\myk,1}+\delta_{\myk,-1}\right)}{1+\epsilon D \Delta t (2\pi/L)^4\myk^4}.
\label{eq:app:timestep}
\end{equation}

We now solve an approximation of Equation~\eqref{eq:fkppperiodic1} numerically, whereby only $N$ modes are used.  As such,  we truncate the Fourier expansions such that $|\myk|<N/2$.  Hence, we replace the Fourier transform of $C(\eta,t)$ with the discrete (fast) Fourier transform analogue, to produce the following algorithm:
\begin{enumerate}
\item Set $n=0$.  Start with initial data $C(\eta,t=0)$, and $Q=C^3(\eta,t=0)-C(\eta,t=0)$ perform a discrete Fourier transform to obtain $\mychat_{\myk}^{n}$ and $\widehat{Q}_{\myk}^{n}$
\item Obtain $\widehat{u}_{\myk}^{n+1}$ from Equation~\eqref{eq:app:timestep}.
\item Perform the inverse Fourier transform to obtain $C(\eta,t=(n+1)\Delta t)$ and hence, 
\[
Q=u^3(\eta,t=(n+1)\Delta t)-C(\eta,t=(n+1)\Delta t).
\]
\item Increment the counter $n$ and repeat steps 2--3 until the final simulation time is reached.
\end{enumerate}
This is an efficient algorithm, as the differentiation $\partial_\eta^{2p}$ is carried out in Fourier space, where it manifests itself as multiplication ($\partial_\eta^{2p}\rightarrow (-1)^p (2\pi/L)^{2p}\myk^{2p}$).  Equally, the convolution
\[
\widehat{Q}_{\myk}=\int_0^L \mathe^{\imag (2\pi/L)\myk \eta}(u^3-u)\mathd \eta
=\left(\sum_{\myk'}\sum_{\myk''} \mychat_{\myk'}\mychat_{\myk''}\mychat_{\myk-\myk'-\myk''}\right)-\mychat_{\myk}
\]
is carried out in real space, where it manifests itself just as ordinary multiplication, i.e. $Q=u^3-u$.  As such, the algorithm is pseudospectral -- it is not a fully spectral algorithm, as the numerical solution is not computed entirely in terms of the Fourier amplitudes $\widehat{u}_{\myk}$.  Instead, at each timestep, we transform back into real space, where $Q$ is computed highly efficiently.  One then reverts to spectral (Fourier) space for the next timestep.

\begin{figure}
  \subfigure[]{\includegraphics[width=0.48\textwidth]{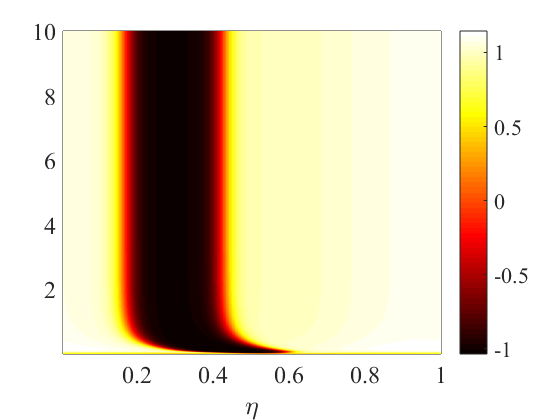}}
  \subfigure[]{\includegraphics[width=0.48\textwidth]{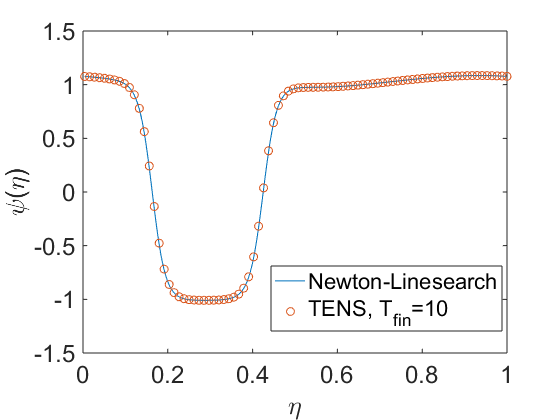}}
  \caption{Sample transiently-evolving numerical simulation (TENS) results for the case $\langle \psi\rangle=0.5$ and $f_0=1.5$.  Also, $v=1$, and $\epsilon=5\times 10^{-4}$.  Panel (a) shows the spacetime evolution of the concentration profile $C(\eta,t)$, up to a final time $T_{\mathrm{fin}}=10$.  Panel (b) shows a snapshot of the concentration at the final time, and a comparison with a steady travelling-wave profile generated with the Newton solver.  A timestep $\Delta t=10^{-4}$ is used.  Also, $N=256$ gridpoints are used in both numerical methods.}
  \label{fig:validation_spacetime}
\end{figure}
A sample implementation of the above pseudospectral algorithm is shown in Figure~\ref{fig:validation_spacetime}.  The initial condition is chosen such that $C(\eta,t=0)=r+\langle \psi\rangle$, where $r\in [-0.1,0.1]$ is a random number generated independently at each spatial position $\eta$.  This represents a fluctuation around the prescribed mean value of the concentration.  The spacetime evolution of the concentration profile is shown in 
Figure~\ref{fig:validation_spacetime}(a).  The solution evolves away from the random initial condition and forms a steady travelling wave (the spacetime evolution is shown in the frame moving with the wave, i.e. in $\eta-t$ variables).  The steady solution agrees exactly with the steady-state solution computed directly via the Newton linesearch method (panel (b)), confirming the correctness of the two distinct numerical methods.



\end{document}